\documentclass[12pt]{article}
\usepackage[utf8]{inputenc}
\usepackage[T2A]{fontenc}
\usepackage{cmap}
\usepackage{indentfirst}
\usepackage[english]{babel}
\usepackage{amsmath,amssymb,amsthm,graphicx}
\usepackage{bbm}
\usepackage[usenames]{color}
\usepackage{colortbl}
\usepackage{caption}
\usepackage{subcaption}
\usepackage[unicode, pdftex]{hyperref}
\usepackage[left=2.5cm,right=2.2cm,top=2.5cm,bottom=3.5cm]{geometry}
\usepackage{wasysym}

\newcommand\R{\mathbb R}

\newtheorem{theorem}{Theorem}

\newtheorem{lemma}{Lemma}

\newtheorem{notation}{Notation}

\begin{document}

\title{On the Packing Density Bound in 3-Space }
\author{Arkadiy Aliev}
\date{}
\maketitle
    \begin{abstract}
        The note shows an easy way to improve E.H. Smith's packing density bound in $\R^3$ from $0.53835...$ to $0.54755...$ .
    \end{abstract}
    \begin{notation}
        Let $K,L\subset\R^{2}$ be two planar convex bodies. We denote their mixed volume by $(K,L)$.
    \end{notation}
    \begin{notation}
        We denote the density of the densest lattice packing by translates of $K$ in $\R^{n}$ by $\delta_{L}(K)$. 
    \end{notation}
    \begin{lemma}
        Let $H\subset \R^{2}$ be a convex centrally symmetric hexagon or a parallelogram. Then for a centrally symmetric convex body $C\subset \R^{2}$ we have:
        $$
        (C,H)\geq \frac{1}{\sqrt{\delta_{L}(C)}}\sqrt{|C|}\sqrt{|H|}.
        $$
    \end{lemma}
    \begin{proof}
        Let $H=A_{1}..A_{6}$ and let $ H' = A'_{1}..A'_{6}$ be the centrally symmetric hexagon circumscribing $C$, which satisfies the condition: $A'_{i}A'_{i+1}\parallel A_{i}A_{i+1}$ for $i = 1,2,3$. Then it is easy to observe that $(C, H) = (H', H)$. Further, since $\delta_{L}(H') = 1$ and $C \subset H'$ we have $\frac{|C|}{|H'|}\leq\delta_{L}(C)$. Therefore
        $$
            (C, H) = (H', H) \geq \sqrt{|H'|}\sqrt{|H|} \geq \frac{1}{\sqrt{\delta_{L}(C)}}\sqrt{|C|}\sqrt{|H|}.
        $$
    \end{proof}
    \begin{theorem} For a centrally symmetric convex body $K\subset \R ^{3}$ we have $\delta_{L}(K)\geq 0.54755...$. 
    \end{theorem}
    \begin{proof}
        We use the notation and construction from \cite{1}. 
        Let $d$ be a chord in $K$ of maximum length. $C:=K\cap d^{\perp}.$ Let $\Lambda\subset\R^{2}$ be a lattice such that $C+\Lambda$ is the densest lattice packing of translates of $C$, $det\Lambda = |P|$, $\frac{|C|}{|P|} = \delta_{L}(C)$. $H_{K}$ is a centrally symmetric hexagon with $|H_{K}|\geq \frac{|P|}{4}$. We may assume that $|H_{K}| = \frac{|P|}{4}$. $h$ is the distance between two layers in 
 the lattice packing of translates of $K$. Then as shown in \cite{1} we have:
        $$
        |K| \geq \frac{|H_{K}|(d-h)}{3} + \frac{h}{3}(|C|+(C,H_{K}) + |H_{K}|). 
        $$
        Therefore since $\delta_{L}(K)\geq \frac{|K|}{h|P|}$ and $\delta_{L}(C) \geq 0.89265... $ we have: 
        $$
        \delta_{L}(K)\geq \frac{1}{12}\frac{d}{h} + \frac{1}{3}\left(\delta_{L}(C) + \frac{1}{2}\sqrt{\delta_{L}(C)}\frac{(C,H_{K})}{\sqrt{|C|}\sqrt{|H_{K}|}}\right) \geq \frac{1}{12} + \frac{1}{3}\left(0.89265... + \frac{1}{2}\right) = 0.54755...
        $$
    \end{proof}

\textit{E-mail address:} \href{mailto:arkadiy.aliev@gmail.com}{arkadiy.aliev@gmail.com} 

\end{document}